\documentclass[11pt]{article}
\usepackage{amsmath}
\usepackage{amsfonts}
\usepackage{graphicx}
\usepackage{latexsym}
\usepackage{url}
\usepackage{color} 
\usepackage{dsfont}
\usepackage{amssymb}
\usepackage{amsthm}
\usepackage{comment}
\usepackage{enumerate}

\excludecomment{en-text}
\includecomment{jp-text}
\excludecomment{comment}

\newcommand{\R}{\mathbb{R}}

\newcommand{\E}{\mathbb{E}}
\newcommand{\PP}{\mathbb{P}}
\newcommand{\N}{\mathbb{N}}

\newcommand{\si}{\sigma}

\newcommand{\ep}{\varepsilon}

\newcommand{\n}{\mathcal N}

\newcommand{\f}{\mathcal F}

\newcommand{\proba}{(\Omega ,\mathcal{F},(\f_t)_{t\in \R},\PP)}

\newcommand{\prob}{(\Omega ,\mathbb{F},\PP)}
\newcommand{\probp}{(\Omega' ,\mathbb{F}',\PP')}

\newcommand{\toop}{\stackrel{\PP}{\longrightarrow}}
\newcommand{\ucp}{\stackrel{u.c.p.}{\longrightarrow}}
\newcommand{\schw}{\stackrel{d}{\longrightarrow}}
\newcommand{\eqschw}{\stackrel{d}{=}}
\newcommand{\stab}{\stackrel{d_{st}}{\longrightarrow}}

\newcommand{\BS}{\overline{S}}
\newcommand{\bS}{\underline{S}}

\newcommand{\bee}{\begin{equation}}
\newcommand{\eee}{\end{equation}}
\newcommand{\bea}{\begin{eqnarray}}
\newcommand{\eea}{\end{eqnarray}}
\newcommand{\bean}{\begin{eqnarray*}}
\newcommand{\eean}{\end{eqnarray*}}

\renewcommand{\theequation}{\arabic{section}.\arabic{equation}}

\newtheorem{prop}{Proposition}[section]

\newtheorem{lem}[prop]{Lemma}
\newtheorem{ex}[prop]{Example}
\newtheorem{theo}[prop]{Theorem}

\begin{document}

\title{Limit theorems for general functionals of Brownian local times \thanks{Nicolas Lengert and Mark
Podolskij 
gratefully acknowledge financial support of ERC Consolidator Grant 815703
``STAMFORD: Statistical Methods for High Dimensional Diffusions''.}}
\author{
Simon Campese \thanks{Department of Mathematics, TU Hamburg, Email: 
simon.campese@tuhh.de}
\and
Nicolas Lengert \thanks{Department of Mathematics, University of Luxembourg, Email: nicolas.lengert@uni.lu.} 
\and
Mark Podolskij  \thanks{Department of Mathematics, University of Luxembourg, Email: mark.podolskij@uni.lu.} 
}

\maketitle

\begin{abstract}

In this paper, we present the asymptotic theory for integrated functions of increments of Brownian local times in space. Specifically, we determine their first-order limit, along with the asymptotic distribution of the fluctuations. Our key result establishes that a standardized version of our statistic converges stably in law towards a mixed normal distribution.
Our contribution builds upon a series of prior works by S. Campese, X. Chen, Y. Hu, W.V. Li, M.B. Markus, D. Nualart and J. Rosen  \cite{C17,CLMR10,HN09,HN10,MR08,R11,R11b}, which delved into special cases of the considered problem. Notably, \cite{CLMR10,HN09,HN10,R11,R11b} explored quadratic and cubic cases, predominantly utilizing the method of moments technique, Malliavin calculus and Ray-Knight theorems to demonstrate asymptotic mixed normality. Meanwhile, \cite{C17} extended the theory to general polynomials under a non-standard centering by exploiting Perkins’  semimartingale representation of local time and the Kailath-Segall formula.
In contrast to the methodologies employed in \cite{CLMR10,HN09,HN10,R11}, our approach relies on infill limit theory for semimartingales, as formulated in \cite{J97, JS03}. Notably, we establish the limit theorem for general functions that satisfy mild smoothness and growth conditions. This extends the scope beyond the polynomial cases studied in previous works, providing a more comprehensive understanding of the asymptotic properties of the considered functionals.  

\end{abstract}


\section{Introduction} \label{sec0}
\setcounter{equation}{0}
\renewcommand{\theequation}{\thesection.\arabic{equation}}

Over the past five decades, the mathematical literature has witnessed a surge in interest regarding the probabilistic and statistical properties of local times. Originating from the structure of a Hamiltonian in a specific polymer model, numerous investigations have been dedicated to the asymptotic theory concerning functionals derived from the local time of a Brownian motion. A notable body of work in this domain includes \cite{C17,CLMR10,HN09,HN10,MR08,R11}. 
Recall that the local time $(L^x)_{x\in \R}$ of a Brownian motion $(W_t)_{t\in [0,1]}$ over a time interval $[0,1]$ is defined as the almost sure limit
\bee \label{loctime}
L^x:= \lim_{\ep \downarrow 0} \frac{1}{2\ep} \int_0^1 1_{(x-\ep,x+\ep)} (W_s) ds.
\eee  
The primary focus of our paper centers around statistics of the form:
\bee \label{Defvf}
V(f)^h_{\R}:= \int_{\R} f\left(h^{-1/2} (L^{x+h} - L^x)  \right) dx, \qquad h>0,
\eee
where $f:\R \to \R$ is a smooth enough function with $f(0)=0$.
Our objective is to ascertain the asymptotic behavior of the statistic $V(f)^h_{\R}$ as $h\to 0$. The theorem below summarizes several special cases, extensively explored in the existing literature, that fall within the scope of our investigation.

\begin{theo} \label{ThSummary}
Let $Z$ be a standard Gaussian random variable independent of the local time  $(L^x)_{x\in \R}$. 
\begin{itemize}
\item[(i)] Chen et al. \cite{CLMR10}, Hu and Nualart \cite{HN09}, Rosen \cite{R11b}, case $f(x)=x^2$: As $h\to 0$
\[
\frac{1}{h^{3/2}} \left( \int_{\R} \left(L^{x+h} - L^x\right)^2 dx -4h \right) \schw \sqrt{\frac{64}{3}\int_{\R} (L^x)^2 dx} \times Z.
\] 
\item[(ii)] Hu and Nualart \cite{HN10}, Rosen \cite{R11}, case $f(x)=x^3$: As $h\to 0$
\[
\frac{1}{h^{2}} \int_{\R} \left(L^{x+h} - L^x\right)^3 dx  \schw \sqrt{192\int_{\R} (L^x)^3 dx} \times Z.
\] 
\item[(iii)] Campese \cite{C17},  case $f(x)=x^q$ with $q\in \N_{\geq 2}$: As $h\to 0$
\[
\frac{1}{h^{3/2}} \left( \int_{\R} \left(L^{x+h} - L^x\right)^q dx +R_{q,h} \right) \schw c_q \sqrt{ \int_{\R} (L^x)^q dx} \times Z,
\] 
where the random variable $R_{q,h} $ is given by 
\[
R_{q,h}:= \sum_{k=1}^{\lfloor q/2 \rfloor} a_{q,k} \int_{\R} \left( L^{x+h} - L^x\right)^{q-2k} \left(4 \int_x^{x+h} L^u du\right)^k dx,
\]
and the constants $a_{q,k}$ and $c_k$ are defined as
\[
a_{q,k}= \frac{(-1)^k q!}{2^k k! (q-2k)!} \qquad \text{and} \qquad c_q = \sqrt{\frac{2^{2q+1} q!}{q+1}}.
\]
\end{itemize}
\end{theo}  

\noindent The weak convergence established in Theorem \ref{ThSummary}(i) for the quadratic case has been demonstrated via the method of moments in \cite{CLMR10}. On the other hand, in \cite{HN09}, techniques from Malliavin calculus, along with a version of the Ray-Knight theorem, were employed to derive the same result. Similar methodologies were applied in \cite{HN10, R11} to establish the cubic case outlined in Theorem \ref{ThSummary}(ii). The more general outcome of \cite{C17}, as presented in Theorem \ref{ThSummary}(iii), employs a distinct technique to establish asymptotic mixed normality. The starting point in \cite{C17} is the semimartingale representation of the local time $(L^x)_{x\in \R}$, initially proven by Perkins in \cite{P82}. This representation, in turn, implies a semimartingale decomposition of the statistic $\int_{\R} \left(L^{x+h} - L^x\right)^q dx$. The somewhat intricate standardization $R_{q,h}$ is derived from the Kailath-Segall formula \cite{SK76}, ensuring that the normalized object is a martingale. In the final step, the asymptotic Ray-Knight theorem is applied to deduce weak convergence.

It is worth noting that Theorem \ref{ThSummary}(iii) extends the results of Theorem \ref{ThSummary}(i) and (ii) due to $R_{2,h}=-4h$ and $R_{3,h}=0$ (cf. \cite{C17}). However, in other cases, the standardization $R_{q,h}$ is somewhat unnatural as it depends on the parameter $h$. Our main result, presented below, not only extends Theorem \ref{ThSummary} to general functions but also employs a much more natural standardization. Unless stated otherwise, all random variables are defined on a given probability space $(\Omega, \mathcal{F},\PP)$.

\begin{theo} \label{mainth}
Let $f\in C(\R)$ be a function with polynomial growth satisfying $f(0)=0$. Define the quantity 
\bee \label{defrho}
\rho_{u}(f):= \E[f(\n(0,u^2))] \qquad \text{for } u\in \R. 
\eee
Then it holds that 
\bee \label{LawLN}
V(f)^h_{\R} \toop V(f)_{\R}:= \int_{\R} \rho_{\si_u}(f) du \qquad \text{where} \qquad \si_u:= 2\sqrt{L^u}
\eee
as $h\to 0$.
If moreover  $f\in C^1(\R)$ and $f,f'$ have polynomial growth we deduce the stable convergence
\bee \label{stabresult}
U(f)^h_{\R}:= h^{-1/2} \left( V(f)^h_{\R} - V(f)_{\R}\right) \stab U(f)_{\R}:= \int_{\R} \sqrt{v^2_{\si_u} - \si^2_u \rho^2_{\si_u}(f')} ~dW'_u,
\eee
where $W'$ is a Brownian motion defined on an extended probability space and independent of $\mathcal F$. The quantity $v_x$ is defined as
\bee \label{vxdef}
v_x^2:= 2 \int_0^1 \text{\rm cov}\left(f(xB_1), f(x(B_{s+1}-B_s)) \right) ds
\eee
with $B$ being a standard Brownian motion.
\end{theo}

\noindent Building on the insights presented in \cite{C17}, our approach begins by leveraging the semimartingale representation of the local time. This transformation allows us to recast the original problem into an asymptotic statistic of a semimartingale. Employing a series of approximation techniques from stochastic analysis, we then apply the limit theory for high-frequency observations of semimartingales, as established in \cite{J97}. This application yields the stable convergence result expressed in \eqref{stabresult}. However, it's important to highlight that our framework diverges from classical results established in works such as \cite{BGJPS06,JP12,KP08,PV10} in several aspects. Firstly, we need to introduce a blocking technique as a necessity to break the correlation in the statistic $V(f)^h_{\R}$. Another notable departure from standard high-frequency theory lies in the assumption regarding the semimartingale property of the diffusion coefficient. This property, crucial for obtaining the necessary smoothness for a stable limit theorem, is absent in our model. Instead, our diffusion coefficient is represented by the process $\si_u$ defined in \eqref{LawLN}, which is not a semimartingale. Consequently, we employ more nuanced techniques to derive the asymptotic theory.

A surprising distinction, in comparison to \cite{KP08}, is observed in the form of the limit at \eqref{stabresult}. Generally, when the function $f$ is not even, the limit typically comprises three terms, revealing an $\mathcal{F}$-conditional bias (cf. \cite{J97,KP08} and Theorem \ref{Thfunc} below). However, in our scenario, we obtain a simpler limit denoted as $U(f)_{\R}$, devoid of an $\mathcal{F}$-conditional bias. This holds true regardless of whether the function $f$ is even or not.

The paper is structured as follows. Section \ref{sec1} provides crucial technical results, including the semimartingale decomposition of the local time $(L^x)_{x\in \R}$ and a functional stable central limit theorem. In Section 3, we delve into the proof of the main result.

\subsection*{Notation}
Unless explicitly stated otherwise, all random variables and stochastic processes are defined on a filtered probability space denoted by $\proba$. All positive constants
are denoted by $C$ (or by $C_p$ if we want to emphasise the dependence on an external parameter $p$) although they may change from line to line. We use the notation
\[
I_t:= [\min(0,t),\max(0,t) ].
\]
We say that a function $f:\R \to \R$ has polynomial growth if it holds that $|f(x)|\leq C(1+|x|^p)$ for some $p\geq0$. We denote by $\langle X,Y \rangle$ the covariation process of two semimartingales $X$ and $Y$.
For real-valued stochastic processes $Y^n$ and $Y$, we employ the notation $Y^n \ucp Y$ to signify uniform convergence in probability, specifically:
\[
\sup_{t\in A} \left|Y^n_t - Y_t \right| \toop 0
\]
for any compact set $A\subset \R$. For a sequence of random variables $(Y^n)_{n\in \N}$ defined on a Polish space $(E,\mathcal{E})$, we say that $Y^n$ converges stably in law towards $Y$ 
($Y^n \stab Y$), which lives on an extension $\probp$ of the original probability space $\prob$,
if and only if
\[
\lim_{n\to \infty} \E[Fg(Y^n)]= \E'[Fg(Y)]
\] 
for all bounded $\mathcal{F}$-measurable random variables $F$ and all bounded continuous functions $g:E\to \R$.

\section{Definitions and preliminary results} \label{sec1}
\setcounter{equation}{0}
\renewcommand{\theequation}{\thesection.\arabic{equation}}
To begin, we utilize the semimartingale representation of the local time process $(L^x)_{x\in\R}$, as derived in \cite{P82}.
This representation posits the existence of a Brownian motion $(B_t)_{t\in \R}$ such that the local time  is expressed as follows:
\begin{align} \label{sem}
L^x = L^z + \int_z^x \sigma_y dB_y +  \int_z^x a_y dy, \qquad x\geq z,
\end{align} 
where the diffusion coefficient $\si$ is defined at \eqref{LawLN}, and the drift coefficient $a$ is a predictable, locally bounded process. This representation, as emphasized in the introduction, serves as a fundamental tool for establishing the stable limit theorem in \eqref{stabresult}. Additionally, we introduce two random times
\begin{align} \label{localset}
\underline{S} := \inf\left\{a\leq 0:~ L^a>0 \right\}, \qquad  
\overline{S} := \sup\left\{a\geq 0:~ L^a>0 \right\},
\end{align} 
and remark that $\underline{S}$ and $\overline{S}$ are stopping times with respect to the filtration generated by $L$. Note $L^x=0$ for any $x \not \in [\bS,\BS]$. 

To establish the results outlined in Theorem \ref{mainth}, it is essential to introduce a functional version of the statistic $V(f)^h_{\R}$. For a fixed $T>0$, we define this functional as follows:
\bee \label{funcV}
V(f)^h_t:= \int_{I_t} f\left(h^{-1/2} (L^{x+h} - L^x)  \right) dx,  \qquad t\in[-T,T]. 
\eee
We obtain the following theorem.

\begin{theo} \label{Thfunc}
Let $f\in C(\R)$ be a function with polynomial growth satisfying $f(0)=0$. Then it holds that 
\bee \label{ucpconver}
V(f)^h \ucp V(f) \text{ as } h\to 0 \qquad \text{where} \qquad  V(f)_t:=  \int_{I_t}  \rho_{\si_u}(f) du.
\eee
Assume moreover that $f\in C^1(\R)$ and $f,f'$ have polynomial growth, and define the process 
$$U(f)^h_t:=h^{-1/2} \left( V(f)^h_t - V(f)_t\right) .$$ 
Then, as $h\to 0$, we obtain the functional stable convergence $U(f)^h \stab U(f)$ on $(C([-T,T]), \|\cdot\|_{\infty})$, where 
\bee
   U(f)_t:= \int_{I_t}  r_{a_x,\si_x} dx + \int_{I_t} w_{\si_x} dB_x + \int_{I_t}  \sqrt{v_{\si_x}^2 - w_{\si_x}^2 } dW'_x. 
\eee
The processes $v_u$ and $W'$ have been introduced in Theorem \ref{mainth}, while the quantities $w_u$ and $r_{u_1,u_2}$ are defined as
\bea \label{quantwr} 
w_u&:=& u\rho_u(f'),  \\[1.5 ex]
r_{u_1,u_2}&:=& u_1 \rho_{u_2}(f') +  \int_0^1 \E\left[f'(u_2(B_{x+1}-B_x)) (B_2^2-2) \right] dx. \nonumber
\eea
\end{theo}

\noindent We now demonstrate that the consistency statement in \eqref{LawLN}, as mentioned in Theorem \ref{mainth}, follows from the more general results provided in Theorem \ref{Thfunc}. Initially, we observe the identities
\bee  \label{identiV}
V(f)^h_{\R}= V(f)^h_{\bS} + V(f)^h_{\BS} + O_{\PP}(h), \qquad V(f)_{\R}= V(f)_{\bS} + V(f)_{\BS} 
\eee
hold. For any $\ep>0$, we conclude that 
\begin{align} 
\PP\left(|V(f)^h_{\bS} - V(f)_{\bS}|>\ep\right) &\leq \PP\left(\sup_{t\in [-T,T]}|V(f)^h_t - V(f)_t|>\ep, |\bS| \leq T   \right) \nonumber \\[1.5 ex] 
& + \PP\left(|\bS| > T\right), \label{Sbounded}
\end{align}
and a similar estimate holds for the probability $\PP(|V(f)^h_{\BS} - V(f)_{\BS}|>\ep)$.
Consequently, the uniform convergence in \eqref{ucpconver} implies the statement in \eqref{LawLN} when we choose $T$ to be sufficiently large and then $h$ to be sufficiently small. This establishes the consistency result in the context of Theorem \ref{mainth}.

Subject to an additional smoothness condition on the function $f$, the expression for the limit $U(f)_t$ simplifies as demonstrated in the following proposition.

\begin{prop} \label{Prop2.2}
Assume that $f(0)=0$,  $f\in C^3(\R)$, and $f$ and its first three derivatives have polynomial growth. Define the function
\bee \label{defG}
 G(u):= \int_0^u \rho_{2\sqrt{x}} (f') dx, \qquad u\geq 0.  
\eee
Then we obtain the identity 
\bee
U(f)_t=G(L^t) - G(L^0) + \int_{I_t} \sqrt{v_{\si_x}^2 - w_{\si_x}^2 } dW'_x.
\eee
\end{prop}

\begin{proof}
First of all, we note that $\rho_u(g)<\infty$ when the function $g$ has polynomial growth. Observing the semimartingale decomposition at \eqref{sem}, an application of It\^o formula gives
\[
F(L^{b}) = F(L^a) + \int_a^b F'(L^u) dL^u + 2  \int_a^b F''(L^u) L^u du,
\]
for any $F\in C^2(\R)$ and any $b>a$. A twofold application of an integration by parts formula implies that 
\[
 \int_0^1 \E\left[f'(u_2(B_{x+1}-B_x)) (B_2^2-2) \right] dx =  u_2^2 \rho_{u_2}(f'''). 
\]
According to definition \eqref{defG} it holds that
\[
G'(u) = \rho_{2\sqrt{u}} (f') \qquad \text{and} \qquad G''(u)=  2\rho_{2\sqrt{u}} (f''').
\]
Consequently, we deduce the identity 
\begin{align*}
\int_{I_t}  r_{a_x,\si_x} dx + \int_{I_t} w_{\si_x} dB_x  &= \int_{I_t}  \rho_{\si_u}(f') dL^u + \int_{I_t}  \si_u^2 \rho_{\si_u}(f''') du \\[1.5 ex]
&= \int_{I_t}  G'(\si^2_u/4) dL^u + \frac{1}{2}\int_{I_t}  \si^2_u G''(\si_u^2/4)du \\[1.5 ex]
&=  \int_{I_t}  G'(L^u) dL^u + 2\int_{I_t}  L^2_u G''(L^u)du = G(L^t) - G(L^0).
\end{align*}
This completes the proof of the proposition. 
\end{proof}

\noindent As a consequence of Proposition \ref{Prop2.2} and the identity $L^x=0$ for $x\not \in [\bS,\BS]$, we infer that 
\bee \label{reducedlimit}
  U(f)_{\R} =  U(f)_{\bS} + U(f)_{\BS} =  \int_{\R} \sqrt{v_{\si_x}^2 - w_{\si_x}^2 } dW'_x,
\eee 
provided the function $f$ satisfies the conditions outlined in Proposition \ref{Prop2.2}. Therefore, the stable convergence asserted in Theorem \ref{mainth} follows from Proposition \ref{Prop2.2} when accompanied by a suitable approximation argument. The details of this argument will be explained in Section \ref{lastsect}.

\begin{ex} \rm
Here, we illustrate that Theorem \ref{mainth} includes the results of Theorem \ref{ThSummary}(i) and (ii) as specific cases. \\ \\
(i) Consider the quadratic case $f(x)=x^2$ and note the identities $\rho_u(f)=u^2$, 
$\rho_u(f')=0$. We immediately conclude that
\[
V(f)_{\R}= \int_{\R} \si_u^2 du = 4 \int_{\R} L^u du =4,
\]
where the last equality follows from the occupation time formula. We also deduce that 
\[
v_x^2 = 2x^4 \int_0^1 \text{cov} \left( B_1^2, (B_{s+1}-B_s)^2\right) ds
= 4x^4 \int_0^1 \text{cov} \left( B_1, B_{s+1}-B_s\right)^2 ds = \frac{4}{3} x^4,
\]
and consequently we get $v^2_{\si_u} = \frac{64}{3} (L^u)^2$. Thus we recover the statement 
of Theorem \ref{ThSummary}(i). \\ \\
(ii) Now, consider the cubic setting $f(x)=x^3$. In this scenario we deduce the identities  
$\rho_u(f)=0$ and $\rho_u(f')=3u^2$. Consequently, $V(f)_{\R}=0$ and $w_u^2= 9 u^6$. A straightforward computation shows that 
\[
v_x^2= 2 \int_0^1 \text{\rm cov}\left(f(xB_1), f(x(B_{s+1}-B_s)) \right) ds =12 x^6. 
\]
Thus we deduce the identity $v^2_{\si_u} - w_{\si_u}^2 =192 (L^u)^3$ and we recover  the statement 
of Theorem \ref{ThSummary}(ii).  \qed
\end{ex}

\section{Proofs} \label{sec2}
\setcounter{equation}{0}
\renewcommand{\theequation}{\thesection.\arabic{equation}}

\subsection{Preliminary results} \label{sec2.1}
To simplify our analysis, we begin by establishing stronger assumptions on the involved stochastic processes. Analogous to the reasoning provided in \eqref{Sbounded}, we can perform all proofs on the set $\{\bS,\BS \in[-T,T]\}$ for some $T>0$. Provided $L^z>0$ for all $z\in[y,x]$, we will also use the 
identity
\begin{align} \label{semsigma}
\sigma_z = \sigma_y + 2(B_z-B_y) + \int_y^z \tilde{a}_u\cdot (L^u)^{-1/2} dy,
\end{align}
where $\tilde{a}$ is a locally bounded process. This identity follows from $\sigma_{x}=2\sqrt{L^x}$
and an application of the It\^o formula to  \eqref{sem}.

Additionally, given that $\tilde{a}, a$ and $\sigma$  are locally bounded processes, we may, without loss of generality, assume that
\begin{align*}
\sup_{\omega \in \Omega,~ x\in [-T,T]} \left(|\tilde{a}_x(\omega)|+|a_x(\omega)| + |\sigma_x(\omega)| \right) \leq C
\end{align*}  
by employing a standard localization argument (cf. \cite{BGJPS06}).

Due to the boundedness of coefficients $a$ and $\sigma$ we deduce from Burkholder inequality 
for any $a<b$ and $p>0$:
\begin{align} \label{lincr}
\E\left[\sup_{x\in [a,b]} |L^x -L^a|^p\right] \leq C_p |b-a|^{p/2}.
\end{align} 
Consequently, due to definition in  \eqref{LawLN}, we also deduce the inequality
\begin{align} \label{sigmaincr2}
\E\left[\sup_{x\in [a,b]}  |\sigma_x -\sigma_a|^p \right] \leq C_p  |b-a|^{p/4}.
\end{align} 
Furthermore, for any function $g:\R \to \R$ with polynomial growth we have that
\bee \label{ginequ}
\E\left[  g\left(h^{-1/2} (L^{x+h} - L^x) \right) \right] \leq C,
\eee
which follows directly from \eqref{lincr}.

We will often use the following lemmata, which are well known results.
\begin{lem} \label{Lemmartin}
Consider the process $Y_t^n= \sum_{i=1}^{\lfloor nt \rfloor} \chi_i^n$, $t\in [0,T]$, where the random variables are $\chi_i^n$ are $\f_{i/n}$-measurable and square integrable. Assume that 
\[
\sum_{i=1}^{\lfloor nt \rfloor} \E\left[\chi_i^n| \f_{(i-1)/n}\right] \ucp Y_t 
\qquad \text{and} \qquad \sum_{i=1}^{\lfloor nT \rfloor} \E\left[(\chi_i^n)^2| \f_{(i-1)/n}\right] \toop 0.
\]  
Then $Y^n \ucp Y$ as $n\to \infty$.
\end{lem} 

\begin{lem} \label{Lemblock}
Consider a sequence of stochastic processes $Y^n$ and $Y^{n,m}$. Assume that 
\begin{align*}
&Y^{n,m} \stab Z^m \quad \text{as } n\to \infty, \qquad Z^m \stab Y \quad \text{as } m\to \infty,
\qquad \text{and}
\\[1.5 ex]
& \lim_{m\to \infty} \limsup_{n\to \infty} \PP\left(\sup_{t\in [0,T]} |Y^{n,m}_t -Y^n_t|>\ep \right)=0
\qquad \text{for any } \ep>0.
\end{align*}
Then it holds $Y^n \stab Y$ as $n\to \infty$.
\end{lem}

The following estimate is important for the mathematical arguments below. 

\begin{prop} \label{prop1}
It holds that 
\[
\sup_{x \in [-T,T]} \mathbb{P}\left(x\in (\bS,\BS),~ L^x \in [0,\varepsilon) \right) \leq C_T \varepsilon.
\]
\end{prop} 

\begin{proof}
Let $\tau_x$ be the first time the Brownian motion $W$ hits the level $x \in \mathbb{R}$. For $x \in (\bS,\BS)$, it must satisfy $\tau_x < 1$.
Now, applying the strong Markov property of Brownian motion, we introduce a new process $\widetilde{W}_t := W_{t + \tau_x} - W_{ \tau_x}$, where $t \geq 0$. Consequently, $\widetilde{W}$ is a new Brownian motion independent of $\tau_x$. Let $L_t^x(W)$ denote the local time of $W$ at point $x$ up to time $t$. This leads to the relation $L_1^x(W) = L_{1-\tau_x}^0(\widetilde{W})$.
A well-known result asserts that $L_{u}^0(\widetilde{W})\eqschw |\widetilde{W}_u|$ for any fixed $u$. 
Thus, by conditioning on $\tau_x$, we infer that
\begin{align*}
\mathbb{P}\left(x\in (\bS,\BS),~ L^x \in [0,\varepsilon) \right) &\leq 
\mathbb{P}\left(\tau_x< 1,~ L^x \in [0,\varepsilon) \right) \\[1.5 ex]
& = \mathbb{P}\left(\tau_x< 1,~ \sqrt{1-\tau_x} \cdot |\mathcal N(0,1)| \in [0,\varepsilon) \right)\\[1.5 ex]
& \leq C \varepsilon \E\left[1_{\{\tau_x<1\}} (1-\tau_x)^{-1/2}\right].
\end{align*}
The density of $\tau_x$ is given by $p(u)= (2\pi)^{-1/2} |x| u^{-3/2} \exp(-x^2/2u)1_{\{u>0\}}$. Hence, we conclude that
\begin{align*}
&\E\left[1_{\{\tau_x<1\}} (1-\tau_x)^{-1/2}\right] < \infty ,
\end{align*}
which completes the proof.
\end{proof}

\subsection{Law of large numbers} \label{sec3}
In this section we show the uniform convergence in probability as stated in \eqref{ucpconver}. An application of \eqref{Sbounded} implies 
the statement \eqref{LawLN}.

The  basic idea of all proofs is to consider the approximation 
$$h^{-1/2} (L^{x+h} - L^x)  \approx h^{-1/2} \sigma_x (B_{x+h} - B_x).$$
Observing this approximation we see that the increments 
$h^{-1/2} (L^{x+h} - L^x)$ and $h^{-1/2} (L^{y+h} - L^y)$ are asymptotically correlated when $|x-y|<h$.
To break this dependence we use a classical blocking  technique. For $i \geq 0$ we introduce the sets 
\begin{align*}
A_i(m) &=[i(m+1)h, i(m+1)h+mh], \\[1.5 ex] 
B_i(m) &=[i(m+1)h+mh, i(m+1)h+(m+1)h].
\end{align*}
Note that the length of $A_i(m)$ is $mh$ (big block) while $B_i(m)$ has the length $h$ (small block). 
In the first step we obtain the following decomposition: 
\begin{align*}
V(f)^h_t = Z_t^{h,m} (f) + R_t^{h,m} (f) + D_t^{h,m} (f),
\end{align*}
where 
\begin{align*}
Z_t^{h,m} (f) &:= \sum_{i\in \N: i(m+1)h+mh \in I_t}  \int_{A_i(m) }
 f\left(h^{-1/2} (L^{x+h} - L^x)  \right) dx, \\[1.5 ex]
 R_t^{h,m} (f) &:= \sum_{i\in \N:~ i(m+1)h+mh\in I_t}  \int_{B_i(m) }
 f\left(h^{-1/2} (L^{x+h} - L^x)  \right) dx,
\end{align*}
and $D_t^{h,m} (f)$ comprises the edge terms and satisfies
\bee \label{edgeconv}
D^{h,m} (f) \ucp 0 \qquad \text{as } h\to 0 
\eee
due to \eqref{ginequ} and the polynomial growth of $f$. Next, we will analyse the asymptotic 
behaviour of the processes $Z^{h,m} (f) $ and $R^{h,m} (f) $. \\ \\ 
(a) \textit{Negligibility of $R_t^{h,m} (f) $:} First, we observe the inequality 
$$\sup_{t \in [-T,T]} |R_t^{h,m} (f)|\leq R_T^{h,m} (|f|) + R_{-T}^{h,m} (|f|)$$ 
Since $f$ has polynomial growth we deduce that $|f(x)|\leq C(1+|x|^p)$ for some
$p>0$. Due to inequality \eqref{lincr} we get 
\[
\E\left[R_T^{h,m} (|f|) + R_{-T}^{h,m} (|f|) \right] \leq C m^{-1}.
\]
Thus, we conclude that
\bee \label{Rconve}
\lim_{m \to \infty} \limsup_{h \to 0}  \PP \left( \sup_{t \in [-T,T]} |R_t^{h,m} (f)| >\ep\right)=0,
\eee
for any $\ep>0$. This proves the negligibility of the term $R^{h,m} (f)$. \qed \\ \\
(b) \textit{Law of large numbers for the approximation:} We introduce the following approximation 
of the statistic $Z_t^{h,m} (f)$:
\begin{align} \label{alphadef}
\overline{Z}_t^{h,m} (f)&:= \sum_{i\in \N:~ i(m+1)h+mh\in I_t}  \alpha_i^h(m), \\[1.5 ex]
\alpha_i^h(m)&:= \int_{A_i(m) }
 f\left(h^{-1/2} \sigma_{t_i^h(m)}(B_{x+h} - B_x)  \right) dx, \nonumber 
\end{align}
where $t_i^h(m)= i(m+1)h$ is the left boundary of the interval $A_i(m)$. Due to Riemann integrability 
we deduce that 
\begin{align*}
\sum_{i\in \N:~ i(m+1)h+mh\in I_t}  \E[\alpha_i^h(m)| \mathcal{F}_{t_i^h(m)}] 
= mh \sum_{i\in \N:~ i(m+1)h+mh\in I_t} \rho_{\sigma_{t_i^h(m)}} (f) \ucp \frac{m}{m+1}V(f)_t
\end{align*}
as $h\to 0$ and $m/(m+1)V(f) \ucp V(f)$ as $m\to \infty$. By Lemma \ref{Lemmartin} it suffices to prove that
\[
\sum_{i\in \N:\leq i(m+1)h+mh\in [-T,T]}  \E\left[ \left|\alpha_i^h(m) \right|^2| \mathcal{F}_{t_i^h(m)}\right]  \toop 0 \quad \text{as } h\to 0.
\]
By \eqref{ginequ} we readily deduce that 
\[
\E[|\alpha_i^h(m)|^2| \mathcal{F}_{t_i^h(m)}]  
\leq C(mh)^2.
\]
Hence, we conclude  
\bee \label{Zhmconv}
\overline{Z}^{h,m} (f) \ucp  \frac{m}{m+1} V(f)  \text{ as } h\to 0, \qquad \text{and} \qquad 
 \frac{m}{m+1} V(f) \ucp  V(f) \text{ as } m\to \infty.
\eee
\qed \\ \\
(c) In view of steps (a) and (b) we are left to proving the statement
\begin{align} \label{finalstep}
\overline{Z}^{h,m} (f) - Z^{h,m} (f) \ucp 0. 
\end{align}
Since $f$ has polynomial growth
we have the following inequality for $\varepsilon,A>0$:
\begin{align*}
|f(x)-f(y)| \leq C\left(w_f(A, \varepsilon) +(1+|x|^p +|y|^p)(1_{\{|x|>A\}} + 1_{\{|y|>A\}}
+ 1_{\{|x-y|>\varepsilon\}}) \right),
\end{align*}
where $w_f(A, \varepsilon):=\sup \{|f(x)-f(y)|:~ |x|,|y|\leq A, ~|x-y|\leq \varepsilon\}$ is the modulus 
of continuity of $f$. Using this inequality and \eqref{lincr}, and also $1_{\{|x|>A\}} \leq A^{-1}|x|$,
$1_{\{|x-y|>\varepsilon\}} \leq \varepsilon^{-1}|x-y|$, we conclude that 
\begin{align*}
&\E\left[\sup_{t\in [-T,T]} \left|\overline{Z}_t^{h,m} (f) - Z_t^{h,m} (f)\right| \right]
\leq C \Big( w_f(A, \varepsilon) + A^{-1} \\[1.5 ex]
&+\varepsilon^{-1} \sum_{i\in \N:~i(m+1)h+mh\in [-T,T]} 
\int_{A_i(m)} \left(h^{1/2} +h^{-1/2} \E\left[\int_{x}^{x+h} |\sigma_u - \sigma_{t_i^h(m)}|^2 du \right]^{1/2} \right) dx.
\end{align*}
Since $\sigma$ is continuous and bounded we see that the third term converges to $0$ 
as $h \to 0$. On the other hand, we have that $\lim_{\varepsilon \to 0} w_f(A, \varepsilon) =0$
for a any fixed $A$.  Hence, we deduce that 
$$\overline{Z}^{h,m} (f) - Z^{h,m} (f) \ucp 0$$ 
by letting first $h\to 0$, then $\varepsilon \to 0$ and $A\to \infty$. 
Due to statements \eqref{Rconve} and \eqref{Zhmconv}, we obtain the convergence in \eqref{ucpconver}.  \qed

\subsection{Stable central limit theorem} \label{sec4}
Demonstrating the stable central limit theorem as stated in Theorem \ref{Thfunc} poses a more intricate challenge. Our approach is primarily based upon limit theorems for semimartingales, notably in works such as \cite{BGJPS06}. It is crucial to highlight that the diffusion coefficient $\sigma_x = 2 \sqrt{L^x}$ is not a semimartingale, introducing a heightened level of complexity to the proofs. We will continue to employ the blocking technique introduced in the preceding section. 

First of all, we decompose our statistic into several terms: 
\bee
U(f)^h = \sum_{k=1}^3 Z^{h,m,k} (f) + \sum_{k=1}^3 R^{h,m,k} (f) + \overline{D}^{h,m}.
\eee
Here the processes $Z^{h,m,k} (f)$, $k=1,2,3$, are big blocks approximations, which 
are defined by
\begin{align*}
&Z_t^{h,m,1} (f):= h^{-1/2} \sum_{i\in \N:~ i(m+1)h+mh \in I_t}  
\left( \alpha_i^h(m) - \E[\alpha_i^h(m)| \mathcal{F}_{t_i^h(m)}] \right), \\[1.5 ex]
&Z_t^{h,m,2} (f):= h^{-1/2} \sum_{i\in \N:~ i(m+1)h+mh\in I_t}  
\left(\int_{A_i(m) }
 f\left(h^{-1/2} (L^{x+h} - L^x)  \right) dx - \alpha_i^h(m) \right) , \\[1.5 ex]
 &Z_t^{h,m,3} (f):= h^{-1/2} \sum_{i\in \N:~i(m+1)h+mh \in I_t}  \int_{A_i(m)} \left(
 \rho_{\sigma_{t_i^h(m)}}(f) - \rho_{\sigma_x}(f)  \right) dx. 
\end{align*}
The small block processes $R^{h,m,k} (f)$, $k=1,2,3$, are introduced in exactly the same way with
the set $A_i(m)$ being replaced by $B_i(m)$ in all relevant definitions.  Finally, the process
$\overline{D}^{h,m}$ comprises all the edge terms. Similarly to the treatment of the term $D^{h,m}$
in \eqref{edgeconv}, we immediately conclude that 
\bee \label{edgeconv2}
\overline{D}^{h,m} \ucp 0 \qquad \text{as } h\to 0. 
\eee
In the following subsections we will show that all small blocks terms are negligible in the sense 
\bee \label{smallbl}
\lim_{m\to \infty} \limsup_{h\to 0} \PP\left(\sup_{t\in [-T,T]} |R_t^{h,m,k} (f)|>\ep \right)=0
\qquad \text{for any } \ep>0,
\eee
for all $k=1,2,3$. Finally, we will show that 
\[
Z^{h,m,1} (f) \stab U^{\prime m}(f), \qquad  Z^{h,m,2} (f) \ucp U^{\prime \prime m}(f),
\qquad Z^{h,m,3} (f) \ucp 0
\]
as $h\to 0$, and moreover
\bea \label{convem2}
&U^{\prime m}(f) \stab U^{\prime m}(f)= \int_0^{\cdot} w_{\si_x} dB_x + 
\int_0^{\cdot} \sqrt{v_{\si_x}^2 - w_{\si_x}^2 } dW'_x, \\[1.5 ex]
&U^{\prime \prime m}(f) \ucp U^{\prime \prime }(f)= \int_0^{\cdot} r_{a_x,\si_x} dx
\nonumber
\eea
as $m\to \infty$. Consequently, due to \eqref{edgeconv2}-\eqref{convem2}, an application of Lemma  \ref{Lemblock} and properties of stable 
convergence imply the statement of Theorem \ref{Thfunc}.

\subsubsection{Central limit theorem for the approximation}  
Recalling the notation from the previous subsection,
 we set 
 \begin{align*} 
Z_t^{h,m,1} (f)&=: 
\sum_{i\in \N:~ i(m+1)h+mh \in I_t} X_i^h(m) .
\end{align*}
We now prove the stable central limit theorem for $Z^{h,m,1} (f)$ as $h\to 0$.
According to Theorem \cite[Theorem IX.7.28]{JS03} we need to show that 
\begin{align} 
\label{cond1} 
&\sum_{i\in \N:~ i(m+1)h+mh \in I_t} \E[ |X_i^h(m)|^2 |  \mathcal{F}_{t_i^h(m)} ] \toop \int_0^t 
v_{\sigma_x}^2(m) dx \\[1.5 ex]
\label{cond2} 
&\sum_{i\in \N:~ i(m+1)h+mh \in I_t} \E[ X_i^h(m) (B_{t_i^h(m) +(m+1)h} 
- B_{t_i^h(m)}) |  \mathcal{F}_{t_i^h(m)} ] \toop c_m \int_0^t 
w_{\sigma_x} dx  \\[1.5 ex]
\label{cond3}
&\sum_{i\in \N:~ i(m+1)h+mh \in I_t} \E[ |X_i^h(m)|^2 1_{\{|X_i^h(m)|> \epsilon\}}|  \mathcal{F}_{t_i^h(m)} ] \toop 0 \quad \forall \epsilon>0 \\[1.5 ex]
\label{cond4} 
&\sum_{i\in \N:~ i(m+1)h+mh \in I_t} \E[ X_i^h(m) (N_{t_i^h(m) +(m+1)h} 
- N_{t_i^h(m)}) |  \mathcal{F}_{t_i^h(m)} ] \toop 0
\end{align}
where the last statement should hold for all bounded continuous martingales $N$ with $\langle B,N \rangle=0$, $c_m=m/(m+1)$, and the function
$v_u(m)$  will be introduced below. \\

\noindent We start by showing the condition \eqref{cond1}. A straightforward
computation using the substitution $x=hz_1, y=hz_2$ shows that   
\begin{align*}
&\E[ |X_i^h(m)|^2 |  \mathcal{F}_{t_i^h(m)} ] = h^{-1} \int_{A_i^2(m)} 
\left(\E\left[f\left(h^{-1/2} \sigma_{t_i^h(m)}(B_{x+h} - B_x)  \right)  \right. \right. \\[1.5 ex]
&\left. \left. \left. \times f\left(h^{-1/2} \sigma_{t_i^h(m)}(B_{y+h} - B_y)  \right)
\right | \mathcal F_{t_i^h(m)} \right] - \rho_{\sigma_{t_i^h(m)}}^2 (f) \right) 1_{\{|x-y|<h\}}dxdy \\[1.5 ex]
&= h \int_{[i(m+1), i(m+1)+m]^2} 
\left(\E\left[f\left( \sigma_{t_i^h(m)}(B_{z_1+1} - B_{z_1})  \right)  \right. \right. \\[1.5 ex]
&\left. \left. \left. \times f\left( \sigma_{t_i^h(m)}(B_{z_2+1} - B_{z_2})  \right)
\right | \mathcal F_{t_i^h(m)} \right] - \rho_{\sigma_{t_i^h(m)}}^2 (f) \right) 1_{\{|z_1-z_2|<1\}}dz_1dz_2.
\end{align*}
Hence, by Riemann integrability we deduce that 
\[
\sum_{i\in \N:~i(m+1)h+mh \in I_t} \E[ |X_i^h(m)|^2 |  \mathcal{F}_{t_i^h(m)} ] \toop \int_0^t 
v_{\sigma_x}^2(m) dx,
\]
where 
\begin{align*}
v_{u}^2(m) &:= \frac{1}{m} \int_{[0,m]^2} \text{cov}
\left(f\left( u(B_{z_1+1} - B_{z_1})  \right),
 f\left( u(B_{z_2+1} - B_{z_2})  \right) \right) 
\\[1.5 ex]
&\times  1_{\{|z_1-z_2|<1\}}dz_1dz_2.
\end{align*}
We note that
\bee
\lim_{m\to \infty} v_{u}^2(m) = v_{u}^2,
\eee
where $v_{u}^2$ has been introduced in  \eqref{vxdef}. \\

\noindent In the next step we show condition \eqref{cond2}. By the integration by parts formula 
we deduce the identity
\begin{align*}
&\E[ X_i^h(m) (B_{t_i^h(m) +(m+1)h} 
- B_{t_i^h(m)}) |  \mathcal{F}_{t_i^h(m)} ] \\[1.5 ex]
&=  \int_{A_i(m) } \E\left[ 
 f\left(h^{-1/2} \sigma_{t_i^h(m)}(B_{x+h} - B_x)  \right) h^{-1/2} 
 (B_{x+h} - B_x)   |  \mathcal{F}_{t_i^h(m)} \right]dx \\[1.5 ex]
 &= (mh)^{-1} w_{\sigma_{t_i^h(m)}}
\end{align*}
where the function $w_u$ has been defined in \eqref{quantwr}. This implies condition 
\eqref{cond2} by Riemann integrability. \\

\noindent To show condition \eqref{cond3}, we observe the inequality 
\[
\E\left[ |X_i^h(m)|^2 1_{\{|X_i^h(m)|> \epsilon\}}|  \mathcal{F}_{t_i^h(m)} \right]
\leq  \epsilon^{-2} \E\left[ |X_i^h(m)|^4 |  \mathcal{F}_{t_i^h(m)} \right] \leq C \epsilon^{-2}m^4 h^2.
\]
Hence, we deduce the statement of \eqref{cond3}. \\

\noindent To prove condition \eqref{cond4}, we apply a martingale representation theorem to deduce the representation
\[
X_i^h(m) = \int_{A_i(m)} \eta_{i,x}^{h,m} dB_x,
\]
where $\eta_{i}^{h,m}$ is a predictable square integrable process. Now, applying It\^o isometry, we obtain that
\[
\E\left[ X_i^h(m) (N_{t_i^h(m) +(m+1)h} 
- N_{t_i^h(m)}) |  \mathcal{F}_{t_i^h(m)} \right]= \E\left[  \int_{A_i(m)} \eta_{i,x}^{h,m} d\langle B, N\rangle _x|  \mathcal{F}_{t_i^h(m)} \right]=0
\]
Consequently, we showed condition   \eqref{cond4}. \\

\noindent Now, due to  \eqref{cond1}-\eqref{cond4}, we conclude the stable convergence $Z^{h,m,1} (f) \stab
U^{\prime m}(f)$ as $h\to 0$ with
\[
U^{\prime m}(f)_t :=  c_m\int_{I_t} w_{\si_x} dB_x + 
\int_{I_t}\sqrt{v_{\si_x}^2(m) - c_m^2w_{\si_x}^2 } dW'_x. 
\]
On the other hand, since $c_m \to 1$ and $v_{u}^2(m) \to v_{u}^2$ as $m\to \infty$, we obtain that
\bee \label{convem}
U^{\prime m}(f) \stab U^{\prime }(f)=  \int_{I_{\cdot}} w_{\si_x} dB_x + 
\int_{I_{\cdot}} \sqrt{v_{\si_x}^2 - w_{\si_x}^2 } dW'_x
\eee
as $m\to \infty$.

\subsubsection{Negligibility of the small blocks: the martingale term} 
Here we show that the small block term $R^{h,m,1} (f)$ is negligible. We 
recall
\begin{align*} 
R_t^{h,m,1} (f)= h^{-1/2} \sum_{i\in \N:~ i(m+1)h+mh \in I_t}  
\left( \beta_i^h(m) - \E[\beta_i^h(m)| \mathcal{F}_{t_i^h(m)+mh}] \right),
\end{align*}
where 
\[
\beta_i^h(m):= \int_{B_i(m) }
 f\left(h^{-1/2} \sigma_{t_i^h(m)}(B_{x+h} - B_x)  \right) dx.
\]
Since $R^{h,m,1} (f)$ is a martingale, $f$ has polynomial growth and $\sigma$
is bounded, we conclude that 
\[
\E\left[|R_T^{h,m,1} (f)|^2 +|R_{-T}^{h,m,1} (f)|^2 \right]\leq C m^{-1}.
\]
Hence, by Lemma \ref{Lemmartin} we obtain that 
\[
\lim_{m\to \infty} \limsup_{h\to 0} \PP\left(\sup_{t\in [-T,T]} |R_t^{h,m,1} (f)|>\ep \right)=0
\qquad \text{for any } \ep>0.
\]

\subsubsection{Riemann sum approximation error} 
We now consider the Riemann sum approximation error
associated with big blocks. We need to show that 
\begin{align*}
Z_t^{h,m,3} (f)=h^{-1/2} \sum_{i\in \N:~ i(m+1)h+mh \in I_t}  \int_{A_i(m)} \left( \rho_{\sigma_x}(f) 
- \rho_{\sigma_{t_i^h(m)}}(f) \right) dx \ucp 0.
\end{align*}
(The corresponding statement for the small block term $R^{h,m,3} (f)$ is shown in exactly the same way). For this purpose we introduce the threshold
\begin{align} \label{choice}
\varepsilon_h = h^r \qquad \text{for some } r \in (1/4,1/2). 
\end{align}
On each big block $A_i(m)$, we will distinguish two cases according to whether  
$L^{t_i^h(m)}<\varepsilon_h$ or $L^{t_i^h(m)}\geq\varepsilon_h$.

We start with the first case. Since $f \in C^1(\R)$
the map $u \mapsto \rho_u(f)$ is $C^1$. Also note that $\sup_{u\in A} |\rho'_u(f)|$ is bounded if $A$
is a compact set. Due to mean value theorem and boundedness of $\sigma$ we have
\begin{align*}
&1_{\left\{t_i^h(m) \in (\bS,\BS), ~L^{t_i^h(m)}<\varepsilon_h \right\}}  \int_{A_i(m)} \left| \rho_{\sigma_x}(f) 
- \rho_{\sigma_{t_i^h(m)}}(f) \right| dx \\[1.5 ex]
& \leq C 
1_{\left\{t_i^h(m) \in (\bS,\BS), ~L^{t_i^h(m)}<\varepsilon_h \right\}}   \int_{A_i(m)} \left|\sigma_x
- \sigma_{t_i^h(m)}\right| dx
\end{align*}
Now, we use Proposition \ref{prop1},  inequality \eqref{sigmaincr2}
as well as H\"older inequality with conjugates $p,q>1$, $1/p+1/q=1$, to deduce that
\[
\E\left[1_{\left\{t_i^h(m) \in (\bS,\BS), ~L^{t_i^h(m)}<\varepsilon_h \right\}}
\left|\sigma_x- \sigma_{t_i^h(m)}\right| \right] \leq C h^{1/4} \varepsilon_h^{1/q}.
\]
Thus, we obtain that
\begin{align} 
&h^{-1/2} \E\left[ \sum_{i\in \N:~i(m+1)h+mh \in [-T,T]}
1_{\left\{t_i^h(m) \in (\bS,\BS), ~L^{t_i^h(m)}<\varepsilon_h \right\}} \int_{A_i(m)} \left| \rho_{\sigma_x}(f) 
- \rho_{\sigma_{t_i^h(m)}}(f) \right| dx \right] \nonumber \\[1.5 ex]
& \leq C h^{-1/4}  \varepsilon_h^{1/q} \to 0, \label{conveZ1}
\end{align}
where we use the definition at \eqref{choice} and choose $q$ close enough to $1$.  

Now, we treat the case $L^{t_i^h(m)} \geq \varepsilon_h$. For a fixed $m\in \N$, we conclude by Borel–Cantelli lemma
and  $\ep_h=h^r$ for $r<1/2$
that there exists a $h_0>0$ such that 
$\PP$-almost surely
$\sup_{x\in A_i(m)}|L^x -L^{t_i^h(m)}| \leq \ep_h/2$  for any $h<h_0$. In the scenario
$L^{t_i^h(m)} \geq \varepsilon_h$ the latter implies that 
\bee \label{lowebo}
\inf_{x\in A_i(m)}|L^x| \geq \ep_h/2 \qquad \PP \text{-almost surely},
\eee
for $h<h_0$.
We introduce the following process:
\begin{align*}
&Z_t^{h,m,3.1} :=h^{-1/2} \sum_{i\in \N:~ i(m+1)h+mh \in I_t}  
1_{\left\{L^{t_i^h(m)} \geq \varepsilon_h \right\}}\int_{A_i(m)} \left( \rho_{\sigma_x}(f) 
- \rho_{\sigma_{t_i^h(m)}}(f) \right) dx.
\end{align*}
To handle the process $Z^{h,m,3.1}$ we will apply the decomposition  \eqref{semsigma}.
First of all, we use the mean value theorem to deduce that
\[
\rho_{\sigma_x}(f) 
- \rho_{\sigma_{t_i^h(m)}}(f) = \rho'_{\sigma_{t_i^h(m)}}(f) (\sigma_x - \sigma_{t_i^h(m)})
+  (\rho'_{\sigma_{x_i^h}}(f) - \rho'_{\sigma_{t_i^h(m)}}(f)) (\sigma_x - \sigma_{t_i^h(m)}),
\]
where $x_i^h$ is a certain point in the interval $(t_i^h(m), x)$. Now, applying \eqref{semsigma}, we decompose $Z^{h,m,3.1}=
Z^{h,m,3.2} +Z^{h,m,3.3} +Z^{h,m,3.4} $ as
\begin{align*}
Z_t^{h,m,3.2}&:= 2h^{-1/2} \sum_{i\in \N:~i(m+1)h+mh\in I_t}
1_{\left\{L^{t_i^h(m)} \geq \varepsilon_h \right\}}  \int_{A_i(m)} 
 \rho'_{\sigma_{t_i^h(m)}}(f) \left(B_x-B_{\sigma_{t_i^h(m)}}\right)dx , \\[1.5 ex]
Z_t^{h,m,3.3}&:= h^{-1/2} \sum_{i\in \N:~i(m+1)h+mh\in I_t} 
1_{\left\{L^{t_i^h(m)} \geq \varepsilon_h \right\}} \int_{A_i(m)} 
 \rho'_{\sigma_{t_i^h(m)}}(f)\left(\int_{t_i^h(m)}^x \tilde{a}_y\cdot (L^y)^{-1/2} dy \right) dx , 
 \\[1.5 ex] 
 Z_t^{h,m,3.4}&:= h^{-1/2} \sum_{i\in \N:~i(m+1)h+mh\in I_t} 
 1_{\left\{L^{t_i^h(m)} \geq \varepsilon_h \right\}} \int_{A_i(m)} 
(\rho'_{\sigma_{x_i^h}}(f) - \rho'_{\sigma_{t_i^h(m)}}(f)) (\sigma_x - \sigma_{t_i^h(m)}) dx. 
\end{align*} 
Since  $Z^{h,m,3.2}$ is a martingale, we obtain  that 
\[
\E\left[ |Z^{h,m,3.2}_{T}|^2 +  |Z^{h,m,3.2}_{-T}|^2\right] \leq C_T h.
\] 
By Lemma \ref{Lemmartin} we deduce that
\bee  \label{Zhm32}
 Z^{h,m,3.2} \ucp 0 \qquad \text{as } h\to 0. 
\eee
Due to \eqref{lowebo} we know that 
$(L^x)^{-1/2} < 2\varepsilon_h^{-1/2}$ for all $x\in A_i(m)$ and $h<h_0$. Since
$\sigma, \tilde{a}$ are bounded, we conclude that
\bee \label{Zhm33}
\E\left[\sup_{t \in [-T,T]} |Z_t^{h,m,3.3}|\right] \leq C_T h^{1/2} \varepsilon_h^{-1/2} \to 0 \qquad \text{as } h\to 0.
\eee 
To handle the last term $Z^{h,m,3.4}$ we apply a similar technique as in step (c) of Section \ref{sec3}. Notice that the quantity $\rho'_{\sigma_{x}}(f)$ is bounded, because 
$\sigma$ is bounded. We obtain that 
\[
\left|\rho'_{\sigma_{x_i^h}}(f) - \rho'_{\sigma_{t_i^h(m)}}(f) \right| \leq
C\left(w_{\rho'(f)}(A, \epsilon) +
 1_{\{ |\sigma_{x_i^h}-\sigma_{t_i^h(m)}|>\epsilon \}} \right).
\]
Since $L^x \geq \varepsilon_h/2$ for all $x\in A_i(m)$ and $h<h_0$, we deduce from representation
\eqref{semsigma} that 
\[
\E\left[|\sigma_x - \sigma_{t_i^h(m)}|^p\right] \leq C\left(h^{p/2} + h^p \varepsilon_h^{-p/2}\right)
\]
for any $p>0$, for all $x\in A_i(m)$ and $h<h_0$. 
Hence, we now obtain from \eqref{sigmaincr2} that 
\[
\E\left[\sup_{t \in [-T,T]} |Z_t^{h,m,3.4}| \right] \leq C\left(w_{\rho'(f)}(A, \varepsilon)\left(1+ h^{1/2}\varepsilon_h^{-1/2}\right)
+  \epsilon^{-3}
h^{1/2} \right)
\]
Thus, we conclude that
\bee  \label{Zhm34}
 Z^{h,m,3.4} \ucp 0 \qquad \text{as } h\to 0. 
\eee
A combination of \eqref{conveZ1} and \eqref{Zhm32}-\eqref{Zhm34} implies the statement 
\[
Z^{h,m,3} (f) \ucp 0 \qquad \text{as } h\to \infty.
\]
Similarly, $R^{h,m,3} (f) \ucp 0$ as $h\to 0$.

\subsubsection{The terms $Z^{h,m,2} (f)$ and  $R^{h,m,2} (f)$} 
In view of the previous steps, we are left with handling the terms $Z^{h,m,2} (f)$ and  $R^{h,m,2} (f)$. 
We start with the term $Z^{h,m,2} (f)$. First, we consider an approximation of  
$Z^{h,m,2} (f)$ given as 
\begin{align*}
&\overline{Z}_t^{h,m,2} (f) :=h^{-1/2}\sum_{i\in \N:~i(m+1)h+mh\in I_t}  \int_{A_i(m) } \E\left[
 f\left(h^{-1/2} (L^{x+h} - L^x)  \right) \right.\\[1.5 ex]
 &\left. -
 f\left(h^{-1/2} \sigma_{t_i^h(m)} (B_{x+h} - B_x)  \right)| \mathcal F_{t_i^h(m)} \right]dx.
\end{align*}
Applying Lemma \ref{Lemmartin} and following the same arguments as presented in part (c) of Section \ref{sec3}
(see the proof of \eqref{finalstep}),
we deduce that 
\bee
\overline{Z}^{h,m,2} (f) - Z^{h,m,2} (f) \ucp 0 \qquad \text{as } h\to 0.
\eee
We use again the mean value theorem to obtain the decomposition 
\begin{align*}
&\E\left[
 f\left(h^{-1/2} (L^{x+h} - L^x)  \right) -
 f\left(h^{-1/2} \sigma_{t_i^h(m)} (B_{x+h} - B_x)  \right)| \mathcal F_{t_i^h(m)} \right] \\[1.5 ex]
 & = \E\left[ h^{-1/2}f' \left(h^{-1/2} \sigma_{t_i^h(m)} (B_{x+h} - B_x)  \right)
 \left((L^{x+h} - L^x) - \sigma_{t_i^h(m)} (B_{x+h} - B_x)  \right)| \mathcal F_{t_i^h(m)} \right]
 \\[1.5 ex]
 & + \E\left[ h^{-1/2}\left(f'(z_i^h) -f' \left(h^{-1/2} \sigma_{t_i^h(m)} (B_{x+h} - B_x)  \right) \right)
 \right. \\[1.5 ex]
& \left. \times
 \left((L^{x+h} - L^x) - \sigma_{t_i^h(m)} (B_{x+h} - B_x)  \right)| \mathcal F_{t_i^h(m)} \right],
\end{align*}
where $z_i^h$ is a point between $h^{-1/2} \sigma_{t_i^h(m)} (B_{x+h} - B_x) $ and 
$h^{-1/2} (L^{x+h} - L^x)$. As in the previous subsection we need to discuss the cases
$L^{t_i^h(m)} \geq \varepsilon_h$  and 
$L^{t_i^h(m)}<\varepsilon_h$ separately. The easier case 
$L^{t_i^h(m)}<\varepsilon_h$ is handled in exactly the same way as presented in \eqref{conveZ1}, so we focus 
on the scenario $L^{t_i^h(m)} \geq \varepsilon_h$. Similarly to the treatment of $Z^{h,m,3.4}$ we conclude that 
\begin{align*}
&h^{-1/2}\sum_{i\in \N:~i(m+1)h+mh<t}  1_{\{L^{t_i^h(m)} \geq \varepsilon_h\}} \int_{A_i(m) } \E\left[
h^{-1/2}\left(f'(z_i^h) -f' \left(h^{-1/2} \sigma_{t_i^h(m)} (B_{x+h} - B_x)  \right) \right) \right. \\[1.5 ex]
& \left. \times 
  \left((L^{x+h} - L^x) - \sigma_{t_i^h(m)} (B_{x+h} - B_x)  \right)| \mathcal F_{t_i^h(m)} \right]dx
  \ucp 0 \qquad \text{as } h\to \infty.
\end{align*}
Thus, we need to show that 
\begin{align}
\overline{Z}_t^{h,m,2.1} &:=  h^{-1/2}\sum_{i\in \N:~i(m+1)h+mh \in I_t}  1_{\{L^{t_i^h(m)} \geq \varepsilon_h\}} \int_{A_i(m) }
\E\left[ h^{-1/2}f' \left(h^{-1/2} \sigma_{t_i^h(m)} (B_{x+h} - B_x)  \right) \right. \nonumber\\[1.5 ex]
& \left. \times
 \left( \int_x^{x+h} (a_u -a_{t_i^h(m)} )du+\int_{x}^{x+h} \left( \int_{t_i^h(m)}^u \tilde{a}_s
 \cdot (L^{s})^{-1/2} ds \right)  dB_u \right)| \mathcal F_{t_i^h(m)} \right] dx \ucp 0 \nonumber \\[1.5 ex]
 \overline{Z}_t^{h,m,2.2} &:= h^{-1/2}\sum_{i\in \N:~i(m+1)h+mh\in I_t}  1_{\{L^{t_i^h(m)} \geq \varepsilon_h\}} \int_{A_i(m) }
\E\left[ h^{-1/2}f' \left(h^{-1/2} \sigma_{t_i^h(m)} (B_{x+h} - B_x)  \right) \right. \nonumber\\[1.5 ex]
& \label{z4} \left. \times
 \left( ha_{t_i^h(m)}+2\int_{x}^{x+h} (B_u-B_{t_i^h(m)})  dB_u \right)| \mathcal F_{t_i^h(m)} \right] dx
 \ucp \frac{m}{m+1} \int_{I_t} r_{a_x,\si_x} dx,
 \end{align}
 as $h\to 0$. The statement $\overline{Z}^{h,m,2.1} \ucp 0$ is obtained along the lines of the arguments presented in the previous 
 subsection. Finally, observe the identities
 \[
 \E\left[ f' \left(h^{-1/2} \sigma_{t_i^h(m)} (B_{x+h} - B_x)\right) a_{t_i^h(m)} | \mathcal F_{t_i^h(m)} \right] = 
 a_{t_i^h(m)} \rho_{\sigma_{t_i^h(m)}}(f')
 \]
 and 
 \begin{align*}
 &2h^{-1}\int_{A_i(m) }\E\left[ f' \left(h^{-1/2} u(B_{x+h} - B_x)  \right) \int_{x}^{x+h} (B_u-B_{t_i^h(m)})  dB_u \right]dx \\[1.5 ex]
 &= 2h^{-1}\int_{A_i(m) }\E\left[ f' \left(h^{-1/2} u (B_{x+h} - B_x)  \right) \int_{x}^{x+h} (B_u-B_{x})  dB_u \right] dx \\[1.5 ex]
 & = 2mh \int_0^1 \E\left[ f' \left( u (B_{y+1} - B_y)  \right) \int_{y}^{y+1} (B_u-B_{y})  dB_u \right] dy 
  \\[1.5 ex]
 & = 2mh \int_0^1 \E\left[ f' \left( u (B_{y+1} - B_y)  \right) \int_{0}^{2} B_u  dB_u \right] dy \\[1.5 ex]
 & = mh  \int_0^1 \E\left[ f' \left( u (B_{y+1} - B_y)  \right) (B_2^2 -2) \right] dy,
 \end{align*}
 where we used the substitution $x=hy$ and the self-similarity of the Brownian motion. Hence,
 the convergence in \eqref{z4} follows from  Riemann integrability.  
 
Following exactly the same arguments we conclude that   
\[
\lim_{m\to \infty} \limsup_{h\to 0} \PP\left(\sup_{t\in [-T,T]} |R_t^{h,m,2} (f)|>\ep \right)=0
\qquad \text{for any } \ep>0.
\]
This completes the proof of stable convergence $U(f)^h \stab U(f)$.

\subsection{Proof of Theorem \ref{mainth}} \label{lastsect}
Here we prove the statements of Theorem \ref{mainth} via an application of Theorem \ref{Thfunc}
and Proposition \ref{Prop2.2}. Recall that we have already shown the convergence at  \eqref{LawLN}; see \eqref{Sbounded}. Thus we are left to proving the stable central limit theorem presented in 
Theorem \ref{mainth}.

We recall that it suffices to show all convergence results under the restriction $\bS,\BS\in[-T,T]$ for some $T>0$. Theorem \ref{Thfunc} states that $U(f)^h \stab U(f)$ on $(C([-T,T]), \|\cdot\|_{\infty})$. 
Since the mapping $F: [-T,T]^2 \times C([-T,T]) \to \R$ defined as $F((t_1,t_2),H):=H(t_1)+H(t_2)$ 
is continuous, we deduce by the properties of stable convergence and \eqref{identiV}:
\begin{align*}
U(f)_{\R}^h &\stab  U(f)_{\bS} +U(f)_{\BS} \\[1.5 ex]
&= \int_{\R}  r_{a_x,\si_x} dx + \int_{\R} w_{\si_x} dB_x + \int_{\R}  \sqrt{v_{\si_x}^2 - w_{\si_x}^2 } dW'_x,
\end{align*}
under conditions of Theorem \ref{Thfunc}. Our task now boils down to demonstrating that the sum of the first two terms in the limit are equal to zero. This assertion has already been established in \eqref{reducedlimit} under the condition $f(0)=0$, with $f\in C^3(\R)$, and $f$ and its first three derivatives exhibiting polynomial growth. Therefore, our focus shifts to confirming that this statement carries over under the weaker
assumptions of Theorem \ref{mainth}.

Let $f\in C^1(\R)$ be an arbitrary function satisfying the conditions of Theorem \ref{mainth}. Then 
there exists a sequence of functions $(f_n)_{n\geq 1}\in C^3(\R)$ that fulfils  the conditions 
$f_n(0)=0$, 
$$|f_n(x)|+ |f'_n(x)|+|f''_n(x)|+ |f'''_n(x)|\leq C(1+|x|^p) \qquad \text{for some } p>0,$$    
and
\bee \label{uniforconve}
\sup_{x\in A} \left(|f_n(x) - f(x)| + |f'_n(x) - f'(x)| \right) \to 0 \qquad \text{as } n\to 0,
\eee
for any compact set $A \subset \R$. In view of Lemma \ref{Lemblock} it suffices to show that 
\bea \label{condit1}
U(f_n)_{\R} \toop U(f)_{\R} \qquad \text{as } n\to \infty, \qquad \text{and} \\[1.5 ex]
\lim_{n\to \infty} \limsup_{h\to 0} \PP\left(\left|U(f_n)_{\R}^h -U(f)_{\R}^h \right|>\ep \right)
=0 \qquad \text{for any } \ep>0. \label{condit2}
\eea
We start by proving the statement \eqref{condit1}. For this purpose we introduce the notation 
$r_{a_x,\si_x}(f)$, $w_{\si_x} (f)$ and $v_{\si_x}(f)$  to explicitly denote the dependence of these quantities on the function $f$. Since $\bS,\BS\in[-T,T]$ it suffices to prove the convergence
\[
\E\left[\int_{-T}^T  |r_{a_x,\si_x}(f_n) - r_{a_x,\si_x}(f)| +
 |w_{\si_x}(f_n) - w_{\si_x}(f)|  +  |v_{\si_x}(f_n) - v_{\si_x}(f)| dx
 \right] \to 0 
 \]
as $n\to \infty$, to conclude \eqref{condit1}. But the latter follows directly from \eqref{uniforconve}
since the processes $a$ and $\si$ are bounded. 

Now, we show condition \eqref{condit2}. Applying Theorem \ref{Thfunc} we deduce that 
\begin{align*}
\PP\left(\left|U(f_n)_{\R}^h -U(f)_{\R}^h \right|>\ep \right) &\leq 
\PP\left(\sup_{t\in [-T,T]}\left|U(f_n)_{t}^h -U(f)_{t}^h \right|>\ep \right) \\[1.5 ex]
&\to \PP\left(\sup_{t\in [-T,T]}\left|U(f_n)_{t} -U(f)_{t} \right|>\ep \right) 
\qquad \text{as } h\to 0.  
\end{align*}
Using Markov and Burkholder inequalities, and the same arguments as in the proof of \eqref{condit1}, we
obtain that 
\[
\PP\left(\sup_{t\in [-T,T]}\left|U(f_n)_{t} -U(f)_{t} \right|>\ep \right) \to 0 \qquad \text{as } n\to \infty.
\]
Thus we deduce \eqref{condit2}, which completes the proof of Theorem \ref{mainth}.

\end{document}